\documentclass[12pt]{article}

\usepackage{url}
\usepackage{bbm}
\usepackage{mathtools}
\usepackage{amssymb}
\usepackage{amsthm}
\usepackage{empheq}
\usepackage{latexsym}
\usepackage{eurosym}
\usepackage{dsfont}
\usepackage{appendix}
\usepackage{color} 
\usepackage[unicode]{hyperref}
\usepackage{frcursive}
\usepackage[utf8]{inputenc}
\usepackage[T1]{fontenc}
\usepackage{geometry}
\usepackage{multirow}
\usepackage{todonotes}
\usepackage{lmodern}
\usepackage{anyfontsize}
\usepackage{pgfplots}
\usepackage{stmaryrd}
\usepackage{float}
\usepackage{graphicx}
\usepackage{caption}
\usepackage{subcaption}
\usepackage{enumerate}
\usepackage{breakcites}
\graphicspath{ {images/} }
\DeclareMathOperator*{\argmax}{arg\,max}
\DeclareMathOperator*{\argmin}{arg\,min}
\pgfplotsset{compat=newest}

\definecolor{red}{rgb}{0.7,0.15,0.15}
\definecolor{green}{rgb}{0,0.5,0}
\definecolor{blue}{rgb}{0,0,0.7}
\hypersetup{colorlinks, linkcolor={red},citecolor={green}, urlcolor={blue}}
            
\makeatletter \@addtoreset{equation}{section}

\newcommand{\ubar}[1]{\text{\b{$#1$}}}

\newtheorem{theorem}{Theorem}
\newtheorem{theorem2}{Theorem}[section]

\newtheorem*{assumption_sig*}{Assumption $\Ac_\sigma$}
\newtheorem*{assumption_lamb*}{Assumption $\Ac^\lambda$}
\newtheorem*{assumption_max*}{Assumption $\Ac^\text{max}$}
\newtheorem*{assumption_fix*}{Assumption $\Ac^\text{fix}$}

\newtheorem{lemma}[theorem2]{Lemma}
\newtheorem{proposition}[theorem2]{Proposition}

\newtheorem{definition}[theorem2]{Definition}
\newtheorem{remark}[theorem2]{Remark}
\newcommand{\comment}[1]{}
\DeclareUnicodeCharacter{014D}{\=o}
\def \E{\mathbb{E}}

\def \N{\mathbb{N}}

\def \P{\mathbb{P}}

\def \R{\mathbb{R}}

\def\Ac{{\cal A}}

\def\Dc{{\cal D}}

\def\Fc{{\cal F}}

\def\Hc{{\cal H}}

\def\Sc{{\cal S}}

\title{On Z-mean reflected BSDEs}
\author{Joffrey Derchu\footnote{\'Ecole Polytechnique, CMAP; joffrey.derchu@polytechnique.edu},~ Thibaut Mastrolia\footnote{Department of Industrial Engineering and Operations Research, University of California, Berkeley; mastrolia@berkeley.edu}}

\begin{document}

\maketitle
\begin{abstract}
In this paper we provide conditions for the existence of supersolutions to BSDEs with mean-reflections on the $Z$ component. We show that, contrary to BSDEs with mean-reflections on the $Y$ component, we cannot expect a supersolution with a deterministic increasing process $K$. Nonetheless, we give conditions for the existence of a supersolution for a stochastic component $K$ and under various constraints. We formalize some previous arguments on the time-inconsistency of such problems, proving that a minimal supersolution is necessarily a solution in our framework.\\

\noindent{\bf Keywords: Constrained BSDEs, Malliavin calculus}.
\end{abstract}

\section{Introduction}
The aim of this paper is to discuss the existence of supersolutions 
$(Y,Z,K)$ of a backward stochastic differential equation (BSDE in short) of the form
\begin{equation}\label{eq::bsde0}
    Y_t = \xi+\int_t^Tf(s,Y_s,Z_s)ds-\int_t^T Z_s dB_s+K_T-K_t, 0\leq t\leq T
\end{equation}
under \textit{weak} constraints on $Z$, where $B$ is a one-dimensional Brownian motion, and $\xi$ and $f$ are potentially non-Markovian. Backward stochastic differential equations were first introduced by \cite{PARDOUX199055,pardoux1992backward} in a Markovian setting to investigate its  connections with viscosity solutions of  non-linear PDE, and were subsequently extended to a non-Markovian framework and used for the valuation of contingent claims in complete markets in \cite{elkpengquenez}, \cite{antonelli}, or in incomplete markets, see \cite{elkquenez_incomplete}, \cite{foellmer}, \cite{Ansel1994CouvertureDA}, \cite{jacka}. These works was then extended to the problem of hedging a claim under almost-sure constraints on the portfolio. In \cite{cvitanic_karatzas}, the authors showed that the hedging price of the claim under convex constraints is given, if it exists, by a minimal supersolution of \eqref{eq::bsde0}. Also, other works (\cite{Peng1999MonotonicLT}, \cite{cvitanic}, \cite{Peng2007ConstrainedBA}, \cite{peng}) show that such a minimal supersolution can be well approximated by a sequence of penalized unconstrained BSDEs, if a supersolution exists. We can also cite the relatively similar approach used to deal with optimal stopping problems and the pricing of American options via reflected BSDEs, see \cite{elk_kapo}, \cite{ESSAKY2011442}, \cite{hamadenelepeltier}.\\
Nevertheless, few papers actually deal with the existence of a supersolution under those constraints. Some interesting works related to this problem include \cite{bouchard:hal-01065794}, \cite{bouchard2018bsde}. The first paper (\cite{bouchard:hal-01065794}) shows existence of supersolutions if $\phi$ is upper bounded and if the constraint can be written as $Z_t\in K$ for $K$ a convex set of $\R$ containing $0$, by relying on a dual formulation of the problem, while \cite{bouchard2018bsde} shows that, if the constraints are of the form $\ell(Z_t)\leq 0$ with $\ell$ the supremum of linear functions, then the existence of a supersolution can be deduced if some particular singular control problem has a value. By and large the existence of a supersolution is not easy to obtain, and we give in this paper some examples of natural problems in which existence does not hold and penalization methods thus yield diverging results.\\
A possible approach is to relax the constraints by considering constraints in expectation. This idea has been used in \cite{briandElieHu} and \cite{briandCard} for constraints on the law of $Y$ to deal with problems of hedging with constraints on a risk measure or with dynamical systems with constraints on a large number of particles.\\

We observe three main approaches in the literature to study mean-reflected BSDE theory (with constraints on the distribution of $Y$). 
\begin{enumerate}
    \item[(i)] First, fixed point arguments can be used when the process $Z$ plays no role in the generator and the problem can be reformulated as a stopping problem (\cite{djehiche2019meanfield}, \cite{chen2020meanfield}). 
    \item[(ii)] Second, a penalization scheme can be used, but requires some monotonicity property (see also \cite{djehiche2019meanfield}, \cite{briandElieHu}, \cite{chen2020meanfield}).
    \item[(iii)] Third, some papers restrain the problem to deterministic $K$, as in \cite{briandElieHu} or \cite{briandCard}.
\end{enumerate}  

In this paper we address the problem of solving \eqref{eq::bsde0} with law constraints on $Z$ of the form $\mathbb E[\ell(t,Z_t)]\geq 0$ for some constraint function $\ell$. Beyond the particular case (i) implying no dependence of the $Z$ process in the generator of the BSDE, we show in this paper that we cannot follow directly the two other approaches (ii) and (iii). It requires to develop another method to investigate the existence of supersolutions to $Z$-mean reflected BSDE. The approach used in the present paper is based on Malliavin calculus and aims at solving the $Z$-mean reflected BSDE in a smooth class of for the component $K$, which will have a stochastic exponential form.\\

The structure of the paper is the following. We first state the definitions and notations in Section \ref{sec::Prelem}. Then, we turn to the existence of supersolutions to BSDE \eqref{eq::bsde0} with $Z$-mean reflection $\mathbb E[\ell(t,Z_t)]\geq 0$ in Section \ref{sec::exist}. This section contains the mains ideas of the paper. In particular, we show in Section \ref{section:impossibility} that contrary to the case of law constraints on $Y$, deterministic non-decreasing processes $K$ cannot always be used so that neither the approach (ii) nor the approach (iii) above can be used. However, if non-deterministic $K$ processes are allowed, then we give conditions on the moments of the Malliavin derivative of the terminal payoff $\xi$ which ensure the existence of a supersolution, see Section \ref{section:supersol}. Then, we investigate in Section \ref{section:minimal} conditions ensuring the existence of a minimal supersolution to a $Z$ mean reflected BSDE. We follow up on an argument given in \cite{briandElieHu} about the time inconsistency of the problem if non-deterministic $K$ processes are allowed, to fully describe all possible minimal supersolutions in this case. Finally, we give in Section \ref{sec::app} an application.

\section{Preliminaries}\label{sec::Prelem}
\subsection{Spaces and Notations}
We are given a finite horizon $T$ and a complete probability space $(\Omega,\Fc,\P)$ endowed with a Brownian motion $B=(B_t)_{0\leq t\leq T}$ and its usual augmented filtration of  $\{\Fc_t\}_{0\leq t\leq T}$. Let $p>1$, we denote by
\begin{itemize}
    \item $L^2(\Fc_T)$  the set of real valued $\Fc_T-$measurable random variables $\xi$ such that $\mathbb E[|\xi|^p]]<\infty$,
    
     \item $\Sc^p$ the set of real valued $\Fc-$adapted continuous processes $Y$ on $[0, T]$ such that
$$\|Y \|_{\Sc^p} := \E[\sup_{0\leq r\leq T} |Y_r|^p]^{\frac{1}{p}} < \infty,$$
    \item $\Hc^p$ the set of predictable continuous $\R-$valued processes $Z$ s.t. $$\|Z\|_{\Hc^p} :=\E [\int_0^T |Z_r|^pdr]^{\frac{1}{p}}< \infty.$$
    \item $\Ac^2$ is the closed subset of $\Sc^2$ consisting in nondecreasing  processes $K = (K_t)_{0\leq t\leq T}$ with $K_0 = 0$.
\end{itemize}
Let $C^{1,2}(\mathbb R^+\times\mathbb R)$ be the space of continuously differentiable function on $\mathbb R^+$ and twice continuously differentiable function on $\mathbb R$. In addition, we denote by $L^2([0,T])$ the set of Borel measurable square integrable real valued functions on $[0,T]$ and let us consider the following inner product on $L^2([0,T])$
$$\langle f,g\rangle :=\int_0^Tf(t)g(t)dt,\text{ for any Borel measurable functions $(f,g)$ on $[0,T]$},$$
with associated norm $\|{\cdot}\|_{L^2([0,T])}$. 

A \textit{cylindrical functional} is a square integrable random variables $F$ of the form
\begin{equation}
\label{eq:cylindrical}
F=f(W_{t_1},\ldots,W_{t_n}), \quad (t_1,\ldots,t_n) \in [0,T]^n.
\end{equation}
where $f$ is a continuously differentiable and bounded function on $\mathbb R^n$ with $n\in \mathbb N\setminus\{0\}$ with bounded derivatives denoted by $f_{x_i}$.
For any cylindrical functional $F$ of the form \eqref{eq:cylindrical}, the Malliavin derivative $D F$ of $F$ is defined as the following $L^2([0,T])-$valued random variable:
\begin{equation}
\label{eq:DF}
D F:=\sum_{i=1}^n f_{x_i}(W_{t_1},\ldots,W_{t_n}) \textbf{1}_{[0,t_i]}.
\end{equation}
It is then customary to identify $DF$ with the stochastic process $(D_t F)_{t\in [0,T]}$. Denote then by $\mathcal{D}^{1,2}$ the closure of the set of cylindrical functionals with respect to the Sobolev norm $\|\cdot\|_{1,2}$, defined as:
$$ \|F\|_{1,2}:=\E\left[|F|^2\right] + \E\left[\int_0^T |D_t F|^2 dt\right].$$

\subsection{Definitions and settings}
Given a terminal condition $\xi\in L^2(\Fc_T)$, a generator $f:[0,T]\times \Omega\times \mathbb R^2\longrightarrow \mathbb R$ and a constraint $\ell:[0,T]\times \mathbb R\longrightarrow \mathbb R$, we aim at constructing processes $(Y,Z,K)\in \Sc^2\times\Hc^2\times\Ac^2$ to the following BSDE
\begin{equation}\label{eq::bsde}
    Y_t = \xi+\int_t^Tf(s,Y_s,Z_s)ds-\int_t^T Z_s dB_s+K_T-K_t, 0\leq t\leq T
\end{equation}
satisfying  the following unilateral constraint\footnote{The choice of unilateral constraints is aesthetic. Bilateral constraints can be treated the same way.}
\begin{equation}\label{eq::constraint}
    \E[\ell(t,Z_t)]\geq 0, \text{ for any } t\in [0, T].
\end{equation}
\begin{definition}
A triplet of process $(Y,Z,K)\in\Sc^2\times\Hc^2\times\Ac^2$ satisfying \eqref{eq::bsde} under the constraint \eqref{eq::constraint} is called a supersolution. A supersolution $(Y,Z,K)$ such that $K_T=0$ a.s. is called a solution.\vspace{0.5em}

 A supersolution $(Y^*,Z^*,K^*)$ such that, for any other supersolution $(Y,Z,K)$, $Y^*_t\leq Y_t$ for all $t\in[0,T]$, holds almost surely, is called a minimal supersolution.
\end{definition}

In the following we will often consider the following assumption extracted from Proposition 5.3 of \cite{elkpengquenez} ensuring the existence of the $(Y,Z)$ component to BSDE \eqref{eq::bsde} together with the existence of a version of $(D_r Y_t,D_r Z_t)_{0\leq r\leq t\leq T}$.

\paragraph{Assumption $(\mathbf{A})$}
\begin{itemize}
    \item[(i)] $\xi\in\Dc^{1,2}$ such that $\int_0^T \mathbb E[|D_t\xi|^2]dt<\infty$,
    \item[(ii)] $f$ is continuously differentiable in $(y,z)$, with uniformly bounded derivatives denoted by $\partial_Y f$ and $\partial_Z f$ respectively in $y$ and $z$ and for each $(y,z)$, $(t,\omega)\longrightarrow f(t,.,y,z)\in\Dc^{1,2}$ and its derivative denoted by $Df$ admits a progressively measurable version such that $|D_. f|$ is uniformly bounded by some constant $d\in\R^+$,
    \item[(iii)] $\xi\in L^4(\mathcal F_T)$ and $f(t,0,0)\in \mathcal H^4$,
\item[(iv)] for any $(y,y',z,z')$ and any $t\in [0,T]$
\[D_r f(t,\omega,y,z)-D_r f(t,\omega,y',z')|\leq \kappa_r(t,\omega)(|y-y'|+|z-z'|),\]
where $\{\kappa_r(t,\cdot),\; t\in [0,T]\}$ is an $\mathbb R^+-$valued adapted process for almost every $r\in [0,t]$ with $\kappa_r\in \mathcal H^4$ such that $\int_0^T \|\kappa_r\|_4^4dr<\infty$. 
\end{itemize}
Note that the properties $(iii)-(iv)$ can be relaxed in order to get Proposition 5.3 of \cite{elkpengquenez}, see for example \cite{mastroliaPossamaiReveillac}.

\section{Existence results}\label{sec::exist}
In this section we turn to existence of supersolution of BSDE \eqref{eq::bsde} under $Z$-mean reflection \eqref{eq::constraint} under the standing assumption $(\mathbf A)$.
\subsection{No restriction to deterministic $K$ or penalization schemes}\label{section:impossibility}

To illustrate that neither the approach (ii) nor the approach (iii) can be used for the existence of supersolution to $Z$ mean reflected BSDE, we first consider a simple example. Let $f=0$, $\ell(t,.)=Id(\cdot)-\nu$, $\nu\in\R$, $\xi=\int_0^T \lambda(t)dB_t$ for some deterministic function $\lambda$ such that there exists $t_0\in [0,T]$ with $\lambda(t_0)<\nu$. If we assume that there exists a solution $(Y,Z,K)$ to \eqref{eq::bsde} with deterministic component $K$, then for any $t\in [0,T]$ we have 
$\mathbb E[Z_t]=\mathbb E_t[D_t\xi]=\lambda(t)$. Hence, the constraint is not satisfied at time $t_0$. It is not enough to impose a consistency condition on $D_T\xi$ only unlike the case of BSDE with mean constraint on the $Y$ component, see \cite[Assumption $(H_\xi)$]{briandElieHu}. We need to assume at least $\lambda(t)\geq \nu$ for any $t$.\\

As a consequence, we cannot expect a penalization scheme to converge to a supersolution. As a matter of fact, consider the following penalized BSDE \[Y^n_t=\xi+\int_t^T n\E[Z^n_s]^-ds -\int_t^TZ^n_sdB_s, \; n\in\N\,\; \xi\in\Dc^{1,2}.\] This seems to penalize the fact that $\mathbb E[Z^n_s]<0$ when $n$ is large. However the penalization is deterministic so if every term converges this would yield a supersolution with deterministic $K$. Note that $Y^n_t=\E_t[\xi+\int_t^T n\E[Z^n_s]^-ds]=\E_t[\xi]+\phi^n(t)$ for some deterministic function $\phi^n$. In particular $Y^n_t-\phi^n(t)=\E_t[\xi]=\xi-\int_t^T D_s \xi dB_s$ so $Z^n_s=D_s\xi$ and $Y^n_t = \E_t[\xi]+n\int_t^T\E[D_s\xi]^-ds$. Thus $Y^n_0$ diverges if $\int_0^T\E[D_s\xi]^-ds>0$. Hence the penalization scheme diverges.\\

The following proposition provides non trivial examples for which we do not have solutions to BSDE \eqref{eq::bsde} under the constraint \eqref{eq::constraint} restricted to the class of deterministic $K$ components.

\begin{proposition}
Let $\ell$ in $C^{1,2}(\R^+\times \R)$. Let $(Y,Z,K)$ be a supersolution to \eqref{eq::bsde0}. Assume that one of the two following condition is satisfied.
\begin{itemize}
    \item[(a)] $\ell(t,x)=x-\nu$ with $\nu\in\R$ such that 
    \begin{align*}\E[\Gamma_t^TD_t\xi+\int_t^T \Gamma_t^sD_t f(s,Y_s,Z_s)ds]<\nu
    \end{align*}
    for some $0\leq t\leq T$, with $\Gamma_t^s = e^{\int_t^s \partial_Y f(r,Y_r,Z_r)dr+\partial_Z f(r,Y_r,Z_r)dB_r-\frac{\partial_Z f(r,Y_r,Z_r)^2}{2}dr}$.
    \item[(b)]   $D_t\xi$ and $D_t\ell$ are bounded,  $\partial_x\ell\geq \ubar m$ and $c\geq -\frac{\partial_{xx}^2\ell}{(\partial_x\ell)^2}$ with $\ubar m,c\in \mathbb R_+^*$. Assume moreover that either $d=0$ or $\partial_x\ell$ bounded and
    \begin{align*}
        \E\big[\log(\mathbb E_t[e^{\sqrt{c}\ell (e^{\int_t^T\partial_Y f(s,Y_s,Z_s)ds}D_t\xi)}e^{\int_t^T \bar m d\sqrt{c}e^{\int_t^s\partial_Y f(u,Y_u,Z_u)du}ds+\int_t^T\partial_Z f(s,Y_s,Z_s)dB_s-\frac{1}{2}\int_t^T(\partial_Z f(s,Y_s,Z_s))^2ds}])\big]<0
    \end{align*}
    for some $0\leq t\leq T$.
\end{itemize}
Then $K$ cannot be deterministic.
\end{proposition}
\begin{proof}
Let $(Y,Z,K)$ be such a supersolution with $K$ deterministic. Under Assumption $(\mathbf{A})$, up to a change of variable for the $Y$ component, and noting that $f(\cdot,K_T-K_\cdot,0)\in \mathcal H^4$ and $f$ is Lipschitz, we deduce from Proposition 5.3 of \cite{elkpengquenez} that
\begin{equation*}
    D_rY_t = D_r\xi+\int_t^T(D_r f(s,Y_s,Z_s)+\partial_Y f(s,Y_s,Z_s)D_rY_s+\partial_Z f(s,Y_s,Z_s)D_rZ_s)ds -\int_t^T D_rZ_s dB_s
\end{equation*}
for $0\leq r\leq t\leq T$.\vspace{0.3em}

\textit{Case (a).} Note that 
\begin{equation*}
    \mathbb E[Z_t] = \mathbb E[D_tY_t] = \E[\Gamma_t^TD_t\xi+\int_t^T \Gamma_t^s D_tf(s,Y_s,Z_s)ds]<\nu,
\end{equation*}

which contradicts \eqref{eq::constraint}.
\vspace{0.3em}

\textit{Case (b).}
By using It\^o's Lemma we get for any $0\leq r\leq t\leq T$
\begin{align*}
   & \ell(t,e^{\int_r^t\partial_Y f(s,Y_s,Z_s)ds}D_rY_t)\\ &= \ell(T,e^{\int_r^T\partial_Y f(s,Y_s,Z_s)ds}D_r\xi)\\
    &+\int_t^T\big( \partial_x\ell(s,e^{\int_r^s\partial_Y f(u,Y_u,Z_u)du}D_rY_s)e^{\int_r^s\partial_Y f(u,Y_u,Z_u)du}(D_r f(s,Y_s,Z_s)+\partial_Z f(s,Y_s,Z_s)D_rZ_s)\\
    &\hspace{3em}-\partial_t\ell(s,e^{\int_r^s\partial_Y f(u,Y_u,Z_u)du}D_rY_s)-\frac{1}{2}\partial_{xx}^2\ell(s,e^{\int_r^s\partial_Y f(u,Y_u,Z_u)du}D_rY_s)(e^{\int_r^s\partial_Y f(u,Y_u,Z_u)du}D_rZ_s)^2\big)ds\\
    &-\int_t^T \partial_x\ell(s,e^{\int_r^s\partial_Y f(u,Y_u,Z_u)du}D_rY_s) e^{\int_r^s\partial_Y f(u,Y_u,Z_u)du}D_rZ_s dB_s.
\end{align*}

As $\partial_x\ell>0$, we set for any $0\leq r\leq T$ and $t\in [r,T]$
\[\tilde Y_t^r:=  \ell(t,e^{\int_r^t\partial_Y f(s,Y_s,Z_s)ds}D_rY_t),\; \tilde Z_t^r:= \partial_x\ell(t,e^{\int_r^t\partial_Y f(s,Y_s,Z_s)ds}D_r Y_t) e^{\int_r^t\partial_Y f(s,Y_s,Z_s)ds}D_rZ_t .  \]
Hence, 

\begin{align*}
    \tilde Y_t^r &= \ell(T,e^{\int_r^T\partial_Y f(s,Y_s,Z_s)ds}D_r\xi)-\int_t^T \tilde Z_s^r dB_s\\
    &+\int_t^T\big( \partial_x\ell(s,e^{\int_r^s\partial_Y f(u,Y_u,Z_u)du}D_rY_s)e^{\int_r^s\partial_Y f(u,Y_u,Z_u)du}D_r f(s,Y_s,Z_s)-\partial_t\ell(s,e^{\int_r^s\partial_Y f(u,Y_u,Z_u)du}D_rY_s)\\
    &\hspace{2.5em}+ \partial_Z f(s,Y_s,Z_s)\tilde Z_s^r-\frac{1}{2}\frac{\partial_{xx}^2\ell(s,e^{\int_r^s\partial_Y f(u,Y_u,Z_u)du}D_rY_s)}{\partial_x\ell(s,e^{\int_r^s\partial_Y f(u,Y_u,Z_u)du}D_rY_s)^2}|\tilde Z_s^r|^2\big)ds.
\end{align*}

We now denote by $(\overline Y_t^r,\overline Z_t^r)$ the unique solution to the following BSDE

\begin{align*}
    \overline Y_t^r &= \ell(T,e^{\int_r^T\partial_Y f(s,Y_s,Z_s)ds}D_r\xi)\\
 &+\int_t^T\big( \partial_x\ell(s,e^{\int_r^s\partial_Y f(u,Y_u,Z_u)du}D_rY_s)e^{\int_r^s\partial_Y f(u,Y_u,Z_u)du}D_r f(s,Y_s,Z_s)-\partial_t\ell(s,e^{\int_r^s\partial_Y f(u,Y_u,Z_u)du}D_rY_s)\\
 &+ \partial_Z f(s,Y_s,Z_s)\overline Z_s^r+\frac{c}{2}|\overline Z_s^r|^2\big)ds-\int_t^T \overline Z_s^r dB_s.
\end{align*}
Note that, there exists a unique solution to this BSDE which is given by using a Cole-Hopf transform, see for instance \cite{imkeller2010results}, so that
\begin{align*}
    \overline Y_t^r = \frac{1}{\sqrt{c}}\log(\mathbb E_t[&e^{\sqrt{c}\ell(T,e^{\int_r^T\partial_Y f(s,Y_s,Z_s)ds}D_r\xi)}e^{\int_t^T\partial_Z f(s,Y_s,Z_s)dB_s-\frac{1}{2}\int_t^T(\partial_Z f(s,Y_s,Z_s))^2ds}])\\
    &e^{\sqrt{c}\int_t^T\big( \partial_x\ell(s,e^{\int_r^s\partial_Y f(u,Y_u,Z_u)du}D_rY_s)e^{\int_r^s\partial_Y f(u,Y_u,Z_u)du}D_r f(s,Y_s,Z_s)-\partial_t\ell(s,e^{\int_r^s\partial_Y f(u,Y_u,Z_u)du}D_rY_s)\big)ds}
\end{align*}
which is bounded given the assumptions of the theorem.
Using Theorem 2.6 in \cite{kobyl}, we have $\tilde Y_t^r\leq \overline Y_t^r$ for any $r\leq t\leq T$. Taking $r=t$ and recalling that $\tilde Y_t^t= \ell(t,Z_t)$, we get
\begin{align*}\ell(t,Z_t)\leq \frac{1}{\sqrt{c}}\log(\mathbb E_t[&e^{\sqrt{c}\ell(T,e^{\int_t^T\partial_Y f(s,Y_s,Z_s)ds}D_t\xi)}e^{\int_t^T\partial_Z f(s,Y_s,Z_s)dB_s-\frac{1}{2}\int_t^T(\partial_Z f(s,Y_s,Z_s))^2ds}])\\
&e^{\sqrt{c}\int_t^T\big( \partial_x\ell(s,e^{\int_t^s\partial_Y f(u,Y_u,Z_u)du}D_tY_s)e^{\int_t^s\partial_Y f(u,Y_u,Z_u)du}D_t f(s,Y_s,Z_s)-\partial_t\ell(s,e^{\int_r^s\partial_Y f(u,Y_u,Z_u)du}D_rY_s)\big)ds}.
\end{align*}

We finally deduce from our assumptions that $\mathbb E[\ell(t,Z_t)]<0$ for some $t$. Hence, there is no solution with deterministic $K$ to BSDE \eqref{eq::bsde} satisfying the constraint \eqref{eq::constraint}.
\end{proof}

\subsection{Non-trivial supersolutions}\label{section:supersol}
In this section, we build non-trivial supersolutions to the BSDE \eqref{eq::bsde} under the constraints \eqref{eq::constraint}. We seek to extend some existence results given in the literature for almost sure constraints.\\

\subsubsection{Weak \textit{v.s.} almost sure constraints}\label{subsec::weakvsstrong}

Note that in \cite{peng} or \cite{Peng2007ConstrainedBA}, the authors investigate the existence of a solution which satisfies the constraint in $Z_s\in\Gamma_s$, where for all $s\in[0,T]$ $\Gamma_s$ is some convex subset of $\R$, but Example 3.1 in \cite{peng} assumes $0\in \Gamma_s$ together with a specific shape of the generator. In \cite{cvitanic}, the authors also assume $0\in \Gamma_s$ and give other conditions ensuring existence in Section 7, relying either on a boundedness assumption or a convexity assumption on the generator. Again in \cite{Raine-1998} the same assumption on $\Gamma_s$ is made, in the Markovian case, with also some boundedness assumptions. It is actually very simple to build examples where there is no supersolution which can satisfy an almost sure constraint, if $0\not\in\Gamma_s$.\\

Suppose $\xi$ and $f$ be lower bounded by reals $\inf \xi$ and $\inf f$ respectively. Using the arguments of \cite{bouchard2018bsde}, we see that any supersolution satisfying $Z_t\geq 1$ for all $t\in[0,T]$ is lower bounded for all $n\in\N$ by $Y^n$ where $(Y^n,Z^n)\in\Sc^2\times\Hc^2$ is the solution to the BSDE
\begin{align*}
    Y^n_t=\xi+\int_t^T(f(s,Y^n_s,Z^n_s)+ n(1-Z^n_s)^+)ds-\int_t^TZ_sdB_s.
\end{align*}
Let $(\tilde Y^n,\tilde Z^n)\in\Sc^2\times\Hc^2$ be the solution to the BSDE
\begin{align*}
    \tilde Y^n_t=\inf \xi+\int_t^T(\inf  f+n(1-\tilde Z^n_s) )ds-\int_t^T\tilde Z_sdB_s.
\end{align*}
From comparison Theorem, we see that $Y_0^n\geq \tilde Y^n_0\geq \inf \xi+ (n +\inf f)T\underset{n\rightarrow +\infty}{\longrightarrow}+\infty$ so this contradicts the existence of a supersolution satisfying the constraint.\\

In \cite{bouchard2018bsde}, the existence of a minimal supersolution is shown to be a consequence of the finiteness of the value of a particular singular control problem, for constraints of the form $q(t,Z_t)\leq 0$ with $q$ a supremum over linear functions. All those results except \cite{cvitanic} rely on a penalization argument and the classical convergence result of \cite{Peng1999MonotonicLT}.\\

However, taking for example  $f=0$, we note that
\[ Z_t=D_tY_t=D_t\xi - \int_t^T D_t Z_s dW_s.\]
Hence, if $\mathbb E[D_t \xi]\geq 1$ for any $t\in [0,T]$, so there exists a solution to BSDE \eqref{eq::bsde} which satisfies $\mathbb E[Z_t]\geq 1$. As a consequence, we expect that we can obtain existence results of supersolutions to \eqref{eq::bsde} by relaxing the constraint on the $Z$ component.

\subsubsection{Existence results}

Before turning to the main results of this work, we show two technical lemmas.

\begin{lemma}\label{lemma_const}
Let $\mathcal Y,\mathcal Z\in \mathcal S^2\times \mathcal H^2$ and $f:[0,T]\times \Omega\times \mathbb R^2\longrightarrow \mathbb R$ such that the derivative of $f$ with respect to the $z$-component is uniformly bounded \textit{i.e.}  $\|\partial_Z f\|_{\infty}<\infty$. Let
\begin{align*}
    \tilde\Gamma_t^T := e^{\int_t^T \partial_Z f(s,\mathcal Y_s,\mathcal Z_s)dB_s-\frac{\partial_Z f(s,\mathcal Y_s,\mathcal Z_s)^2}{2}ds},
\end{align*}
for any $t\in [0,T]$. Then
\[\begin{cases}

    \E[|\tilde\Gamma_t^T|^p]^{1/p}\geq e^{(-1+p)\frac{\|\partial_Z f\|^2}{2}(T-t)}, & \text{for any } p<0,\\
    \E[|\tilde\Gamma_t^T|^p]^{1/p}\leq e^{(-1+p)\frac{\| |\partial_Z f\|^2}{2}(T-t)} & \text{for any } p>1.
\end{cases}\]
\end{lemma}
\begin{proof}
 Let $q>1$. Then by applying H\"older's Inequality, 
\begin{align*}
    &\E[|\tilde\Gamma_t^T|^p] = \E[e^{\int_t^T p\partial_Z f(s,\mathcal Y_s,\mathcal Z_s)dB_s-p\frac{\partial_Z f(s,\mathcal Y_s,\mathcal Z_s)^2}{2}ds}]\\
    &=\E[e^{\int_t^T p\partial_Z f(s,\mathcal Y_s,\mathcal Z_s)dB_s-\frac{qp^2}{2}\frac{\partial_Z f(s,\mathcal Y_s,\mathcal Z_s)^2}{2}ds+p(-1+qp)\frac{\partial_Z f(s,\mathcal Y_s,\mathcal Z_s)^2}{2}ds}]\\
    &\leq \E[e^{\frac{pq(-1+qp)}{(q-1)}\int_t^T\frac{\partial_Z f(s,y_s,z_s)^2}{2}ds}]^{\frac{q-1}{q}}.
\end{align*}
Choosing $q$ such that $\frac{pq(-1+qp)}{(q-1)}>0$ whether $p<0$ or $p>1$, we have
\[\E[|\tilde\Gamma_t^T|^p] \leq e^{p(-1+qp)\frac{\|\partial_Z f\|^2}{2}(T-t)}.\] We obtain the result by putting this to the power $\frac{1}{p}$ and taking $q\longrightarrow 1$.
\end{proof}

Remember that the process $K$ appearing in the BSDE must be non-decreasing and non-negative. We have seen that we cannot restrict the study to deterministic $K$ as in \cite{briandElieHu}. We thus focus on some particular shapes of $K$ with stochastic exponential similarly to \cite[Remark 6]{briandElieHu}. 
\begin{lemma}
Let $\alpha:[0,T]\mapsto\R$ be a deterministic c\`adl\`ag function with $\int_0^T \alpha_t^2dt<\infty$ and $K_t=\int_0^te^{\int_0^s\{\alpha_rdB_r-\frac{1}{2}\alpha_r^2dr\}}ds$ for $0\leq t\leq T$. Then, $K\in\Ac^2$ and $K_t\in\Dc^{1,2}$ for all $0\leq t\leq T$.
\end{lemma}
\begin{proof}
The process $K$ is well defined as $\E[(e^{\int_0^s\alpha_rdB_r-\frac{1}{2}\alpha_r^2dr})^2]=e^{\int_0^s\alpha_r^2dr}\leq e^{\int_0^T\alpha_r^2dr}$ for any $s\leq T$. Clearly $K_0=0$ and $K$ is non-decreasing. Also
\begin{align*}
    \E[(\sup_{t\in[0,T]}K_t)^2]=\E[K_T^2]\leq\E[(T\sup_{s\in[0,T]}e^{\int_0^s\alpha_rdB_r-\frac{1}{2}\alpha_r^2dr})^2]\leq 4T^2\E[(e^{\int_0^T\alpha_rdB_r-\frac{1}{2}\alpha_r^2dr})^2]<\infty,
\end{align*}
where we have used successively the fact that $K$ is non-decreasing, that $K$ is pathwise differentiable, and the Doob inequality. Consequently, $K\in\Sc^2$. Moreover, we note that $D_uK_t=\mathbf{1}_{u\leq t}\alpha_uK_t$ for all $0\leq u,t\leq T$ so $\E[\int_0^T|D_uK_t|^2du]=\E[K_t^2]\int_0^t \alpha_t^2dt<\infty$ so $K_t\in\Dc^{1,2}$.
\end{proof}

These lemmas enable us to construct supersolutions to \eqref{eq::bsde} with linear constraints and under conditions on the Malliavin derivative of $\xi$.

\begin{proposition}\label{th::exi}
We assume that $\ell(t,.)=Id(\cdot)-\nu(t)$ for some real function $\nu$ on $[0,T]$ with $\nu(T)\leq\E[D_T \xi]$. In addition, assume that one of the following condition is satisfied,

\begin{enumerate}
    \item $\partial_Z f=0$, and
    \begin{align*}
        t\in[0,T)\mapsto\frac{1}{T-t}\big[\nu(t)-\big(-e^{(T-t)\sup \partial_Y f }\E[D_t\xi^-]+e^{(T-t)\inf \partial_Y f }\E[D_t\xi^+]\big)\big]
    \end{align*}
    is upper bounded;\\

\item there exist two map $p_+,p_-$ from $[0,T)$ into $(1,+\infty)$ such that
    \begin{align*}
        t\in[0,T)\mapsto&\frac{1 }{T-t}\big[\nu(t)-(-e^{(T-t)\sup \partial_Y f }e^{\frac{1}{p_-(t)-1}\frac{\sup |\partial_Z f|^2}{2}(T-t)}\E[|D_t\xi^-|^{p_-(t)}]^{\frac{1}{p_-(t)}}\\
        &+e^{(T-t)\inf \partial_Y f }e^{\frac{p_+(t)}{1-p_+(t)}\frac{\sup |\partial_Z f|^2}{2}(T-t)}\E[|D_t\xi^+|^{\frac{1}{p_+(t)}}]^{p_+(t)})\big]
    \end{align*}
    is upper bounded.
\end{enumerate}

Then there exists a supersolution to BSDE \eqref{eq::bsde} under the $Z-$mean constraint \eqref{eq::constraint}.

\end{proposition}
\begin{proof}
Let $\alpha:[0,T]\longrightarrow\R$ some deterministic c\`adl\`ag function, $k_s=e^{\int_0^s\alpha(u)dB_u-\frac{1}{2}\int_0^s\alpha(u)^2du}$, $dK_s=k_sds$. We have $K\in\Ac^2$. Under Assumption $(\mathbf{A})$ and noticing that $f(t,0,0)+k_t\in\Hc^4$, for example by using Minkowski inequality, we get
\begin{equation*}
    D_rY_t = D_r\xi+\int_t^T(D_r f(s,Y_s,Z_s)+\partial_Y f(s,Y_s,Z_s)D_rY_s+\partial_Z f(s,Y_s,Z_s)D_rZ_s)ds -\int_t^T D_rZ_s dB_s+\int_t^T D_rk_sds
\end{equation*}
for $0\leq r\leq t\leq T$, so
\begin{equation*}
    \mathbb E[Z_t] = \mathbb E[D_tY_t] = \E[\Gamma_t^TD_t\xi+\int_t^T \Gamma_t^s D_tf(s,Y_s,Z_s)ds+\int_t^T\Gamma_t^s D_tk_sds]
\end{equation*}
with $\Gamma_t^s = e^{\int_t^s \partial_Y f(r,Y_r,Z_r)dr+\partial_Z f(r,Y_r,Z_r)dB_r-\frac{\partial_Z f(r,Y_r,Z_r)^2}{2}dr}$.\\

\textit{Case 1. No dependence with respect to $Z$ in the driver.}\\

We have
\begin{align*}
    \mathbb E[Z_t]\geq& -e^{(T-t)\sup \partial_Y f }\E[D_t\xi^-]+e^{(T-t)\inf \partial_Y f }\E[D_t\xi^+]-d e^{(T-t)\sup |\partial_Y f| }(T-t) \\
    &+\int_t^T[\alpha(t)e^{(s-t)\inf \partial_Y f }\mathbf{1}_{\alpha(t)\geq 0}+\alpha(t)e^{(s-t)\sup \partial_Y f }\mathbf{1}_{\alpha(t)< 0}]ds\\
    \geq& -e^{(T-t)\sup \partial_Y f }\E[D_t\xi^-]+e^{(T-t)\inf \partial_Y f }\E[D_t\xi^+]-d e^{(T-t)\sup |\partial_Y f| }(T-t) \\
    &+\alpha(t)^+i(t)-\alpha(t)^-s(t)
\end{align*}
where, for $0\leq t\leq  T$, 
\begin{align*}
    s(t)=\begin{cases}
        \frac{e^{(T-t)\sup \partial_Y f }-1}{\sup \partial_Y f }&\text{ if }\sup \partial_Y f\neq 0\\
        T-t&\text{ otherwise}
    \end{cases}
\end{align*}
and
\begin{align*}
    i(t)=\begin{cases}
        \frac{e^{(T-t)\inf \partial_Y f }-1}{\inf \partial_Y f }&\text{ if }\inf \partial_Y f\neq 0\\
        T-t&\text{ otherwise}.
    \end{cases}
\end{align*}
Since $\mathbb E[Z_T]=\mathbb E[D_T\xi]\geq \nu(T)$, it remains to choose $\alpha$ wisely so that the constraint is satisfied on $[0,T)$.
We thus choose
\begin{align*}
    \alpha(t) = \big(\sup_{0\leq t<T}\frac{1}{i(t)}\big[\nu(t)-(-e^{(T-t)\sup \partial_Y f }\E[D_t\xi^-]+e^{(T-t)\inf \partial_Y f }\E[D_t\xi^+]-d e^{(T-t)\sup |\partial_Y f| }(T-t))\big]\big)^+
\end{align*}
for all $0\leq t\leq T$.
Given the assumption and the definition of $i$ we have $0\leq \alpha<\infty$, and, by construction, $\E[Z_t]\geq \nu(t)$.\\

\textit{Case 2. Z dependence in the driver.}\\

Let now $p>1$ and define $k^p_s=e^{\int_0^s\alpha(u)dB_u-\frac{1}{2p}\int_0^s\alpha(u)^2du}$, $dK_s=k^p_sds$ with $\alpha$ being nonnegative. We have $K\in\Ac^2$. Following the lines of the proof of Case 1., we similarly get 
\begin{align*}
 \E[Z_t]\geq& -e^{(T-t)\sup \partial_Y f }\E[\tilde\Gamma_t^T D_t\xi^-]+e^{(T-t)\inf \partial_Y f }\E[\tilde\Gamma_t^T D_t\xi^+]-d e^{(T-t)\sup |\partial_Y f| }(T-t) \\
    &+\int_t^T\E[\tilde\Gamma_t^sD_tk^p_s]e^{(s-t)\inf \partial_Y f }ds
\end{align*}
with $\tilde\Gamma_t^s = e^{\int_t^s \partial_Z f(r,Y_r,Z_r)dB_r-\frac{\partial_Z f(r,Y_r,Z_r)^2}{2}dr}$. 

Applying Hölder's inequality  for some $q>1$, we get 
$$\E[\tilde\Gamma_t^T D_t\xi^-]\leq \E[|\tilde\Gamma_t^T|^\frac{q}{q-1}]^{\frac{1-q}{q}}\E[|D_t\xi^-|^{q}]^{\frac{1}{q}}\leq e^{\frac{1}{q-1}\frac{\sup |\partial_Z f|^2}{2}(T-t)}\E[|D_t\xi^-|^{q}]^{\frac{1}{q}}.$$

Now, from Reverse Hölder inequality together with Lemma \ref{lemma_const} we get for any $q>1$ $$\E[\tilde\Gamma_t^T D_t\xi^+]\geq \E[|\tilde\Gamma_t^T|^\frac{1}{1-q}]^{1-q}\E[|D_t\xi^+|^{\frac{1}{q}}]^q\geq e^{\frac{q}{1-q}\frac{\sup |\partial_Z f|^2}{2}(T-t)}\E[|D_t\xi^+|^{\frac{1}{q}}]^q,$$ 
and 
$$\E[\tilde\Gamma_t^s D_tk^p_s]\geq \E[|\tilde\Gamma_t^s|^\frac{1}{1-p}]^{1-p}\E[|D_tk^p_s|^{\frac{1}{p}}]^p\geq e^{\frac{p}{1-p}\frac{\sup |\partial_Z f|^2}{2}(s-t)}\alpha(t).$$ 
Consequently, for any $p_+,p_-:[0,T]\mapsto (1,+\infty)$,
\begin{align*}
    \E[Z_t]\geq& -e^{(T-t)\sup \partial_Y f }\E[\tilde\Gamma_t^T D_t\xi^-]+e^{(T-t)\inf \partial_Y f }\E[\tilde\Gamma_t^T D_t\xi^+]-d e^{(T-t)\sup |\partial_Y f| }(T-t) \\
    &+\int_t^Te^{\frac{p}{1-p}\frac{\sup |\partial_Z f|^2}{2}(s-t)}\alpha(t)e^{(s-t)\inf \partial_Y f }ds\\
    \geq& -e^{(T-t)\sup \partial_Y f }e^{\frac{1}{p_-(t)-1}\frac{\sup |\partial_Z f|^2}{2}(T-t)}\E[|D_t\xi^-|^{p_-(t)}]^{\frac{1}{p_-(t)}}\\
    &+e^{(T-t)\inf \partial_Y f }e^{\frac{p_+(t)}{1-p_+(t)}\frac{\sup |\partial_Z f|^2}{2}(T-t)}\E[|D_t\xi^+|^{\frac{1}{p_+(t)}}]^{p_+(t)}\\
    &-d e^{(T-t)\sup |\partial_Y f| }(T-t)+\alpha(t)i(t)
\end{align*}
with
\begin{align*}
    i(t)=\begin{cases}
        \frac{e^{(\frac{p}{1-p}\frac{\sup |\partial_Z f|^2}{2}+\inf \partial_Y f)(T-t)}-1}{\frac{p}{1-p}\frac{\sup |\partial_Z f|^2}{2}+\inf \partial_Y f}&\text{ if }(\frac{p}{1-p}\frac{\sup |\partial_Z f|^2}{2}+\inf \partial_Y f)\neq 0\\
        T-t&\text{ otherwise}.
    \end{cases}
\end{align*}
so the same argument applies, with
\begin{align*}
    \alpha(t) = \sup_{0\leq t<T}\frac{1}{i(t)}\big[\nu(t)-(&-e^{(T-t)\sup \partial_Y f }e^{\frac{1}{p_-(t)-1}\frac{\sup |\partial_Z f|^2}{2}(T-t)}\E[|D_t\xi^-|^{p_-(t)}]^{\frac{1}{p_-(t)}}\\
    &+e^{(T-t)\inf \partial_Y f }e^{\frac{p_+(t)}{1-p_+(t)}\frac{\sup |\partial_Z f|^2}{2}(T-t)}\E[|D_t\xi^+|^{\frac{1}{p_+(t)}}]^{p_+(t)}-d e^{(T-t)\sup |\partial_Y f| }(T-t))\big]
\end{align*}
for all $0\leq t\leq T$.
\end{proof}
\begin{remark}
For a constraint of the type $\E[Z_t]\leq \nu(t)$ the assumptions in Proposition \ref{th::exi} become $\nu(T)\geq \E[D_T\xi]$ and
\begin{enumerate}
    \item if $\partial_Z f=0$,
    \begin{align*}
        t\in[0,T)\mapsto\frac{1}{T-t}\big[\nu(t)-(-e^{(T-t)\inf \partial_Y f }\E[D_t\xi^-]+e^{(T-t)\sup \partial_Y f }\E[D_t\xi^+])\big]
    \end{align*}
    lower bounded;
    \item otherwise
    \begin{align*}
        t\in[0,T)\mapsto&\frac{1 }{T-t}\big[\nu(t)-(-e^{(T-t)\inf \partial_Y f }e^{\frac{1}{p_-(t)-1}\frac{\sup |\partial_Z f|^2}{2}(T-t)}\E[|D_t\xi^-|^{p_-(t)}]^{\frac{1}{p_-(t)}}\\
        &+e^{(T-t)\sup \partial_Y f }e^{\frac{p_+(t)}{1-p_+(t)}\frac{\sup |\partial_Z f|^2}{2}(T-t)}\E[|D_t\xi^+|^{\frac{1}{p_+(t)}}]^{p_+(t)})\big]
    \end{align*}
    lower bounded for some functions $p_+,p_-:[0,T)\longrightarrow(1,+\infty)$.
\end{enumerate}

\end{remark}

We now turn to more general constraints. We can derive from the previous proposition sufficient conditions in the case of concave or convex constraints.
\begin{theorem}\label{th::thex}
Let $\ell$ in $C^{1,2}(\R^+\times \R)$, with $\partial_x \ell\geq K\in\R^*_+$ and $-\ubar C\leq \frac{\partial_{xx}^2 \ell}{(\partial_x \ell)^2}\leq C$. Assume $|\partial_t \ell|\leq \bar M\in \R_+$ and either $d=0$ or $\partial_x \ell\leq \bar K\in\R$, and $D_t\xi$ is bounded. Let one of the two following condition be satisfied.
\begin{itemize}
    \item  $\ell$ is convex and the conditions of Proposition \ref{th::exi} associated with the constraint $\E[Z_t]\geq \ell(t,.)^{-1}(0)$ are satisfied;
    \item  $\ell$ is concave with
    $0\leq \E[\ell(T,D_T\xi)]$ and either\\
    \begin{enumerate}
        \item $\partial_Z f=0$, and
                \begin{align*}
                 t\in[0,T)\mapsto\frac{1}{T-t}\big[-\E[\ell(T,-e^{(T-t)\sup \partial_Y f }D_t\xi^-+e^{(T-t)\inf \partial_Y f }D_t\xi^+)]\big]
                \end{align*}
                is upper bounded,\\
                
            or\\
        \item  there exist two maps $p_+,p_-$ from $[0,T)$ into $(1,+\infty)$ such that
            \begin{align*}
                t\in[0,T)\mapsto&\frac{1 }{T-t}\big[-(-e^{\frac{1}{p_-(t)-1}\frac{\sup |\partial_Z f|^2}{2}(T-t)}\E[|\chi_t^-|^{p_-(t)}]^{\frac{1}{p_-(t)}}\\
                &+e^{\frac{p_+(t)}{1-p_+(t)}\frac{\sup |\partial_Z f|^2}{2}(T-t)}\E[|\chi_t^+|^{\frac{1}{p_+(t)}}]^{p_+(t)})\big]
            \end{align*}
            is upper bounded, with $\chi_t=\ell(T,-e^{(T-t)\sup \partial_Y f }D_t\xi^-+e^{(T-t)\inf \partial_Y f }D_t\xi^+)$.
    \end{enumerate}
    \comment{\item $C>0$: there exists some deterministic c\`adl\`ag function $u$ on $[0,T]$ such that
        \begin{align*}
            0\leq&-\bar K de^{(T-t)\sup\partial_Y f}(T-t)+l(-e^{(T-t)\sup|\partial_Y f|}\sup|D_.\xi|)+Ke^{-(T-t)\sup|\partial_Y f|} u(t)(T-t)\\
            &+\frac{1}{2}(Ke^{-(T-t)\sup|\partial_Y f|} u(t))^2(T-t)^2\int_0^tu(r)^2dr-\frac{C}{2}(Ke^{-(T-t)\sup|\partial_Y f|} u(t))^2\int_t^T(T-r)^2u(r)^2dr\\
            &+Ke^{-(T-t)\sup|\partial_Y f|} u(t)(\inf\partial_Z f)\int_t^T(T-r)u(r)dr
        \end{align*}}
\end{itemize}
\end{theorem}
\begin{proof}

If $\ubar C\leq 0$, \textit{i.e.} if $\ell(t,.)$ is convex, then $\E[\ell(t,Z_t)]\geq \ell(t,\E[Z_t])$. As $\ell(t,.)$ is strictly increasing and continuous the constraint $\E[\ell(t,Z_t)]\geq 0$ is satisfied if $\E[Z_t]\geq \ell(t,.)^{-1}(0)$. This can be done under the assumptions of Proposition \ref{th::exi}.\\

In the following we assume $\ubar C\geq 0$. Let $\alpha:[0,T]\longrightarrow\R$ some non-negative deterministic c\`adl\`ag function, $k_s=e^{\int_0^s\alpha(u)dB_u-\int_0^s\frac{\alpha(u)^2}{2}du}$, $dK_s=k_sds$. We have $K\in\Ac^2$. Under Assumption $(\mathbf{A})$ and noticing that $f(t,\omega,0,0)+k_t\in\Hc^4$ we get
\begin{equation*}
    D_rY_t = D_r\xi+\int_t^T(D_r f(s,Y_s,Z_s)+\partial_Y f(s,Y_s,Z_s)D_rY_s+\partial_Z f(s,Y_s,Z_s)D_rZ_s)ds -\int_t^T D_rZ_s dB_s+\int_t^T D_rk_sds
\end{equation*}
for $0\leq r\leq t\leq T$, so, using Itô's Lemma we get
\begin{align*}
    &\ell(t,e^{\int_r^t\partial_Y f(s,Y_s,Z_s)ds}D_rY_t) = \ell(T,e^{\int_r^T\partial_Y f(s,Y_s,Z_s)ds}D_r\xi)\\
    &+\int_t^T\big( \partial_x \ell(s,e^{\int_r^s\partial_Y f(u,Y_u,Z_u)du}D_rY_s)e^{\int_r^s\partial_Y f(u,Y_u,Z_u)du}(D_r f(s,Y_s,Z_s)+\partial_Z f(s,Y_s,Z_s)D_rZ_s)\\
    &-\partial_x \ell(s,e^{\int_r^s\partial_Y f(u,Y_u,Z_u)du}D_rY_s)-\frac{1}{2}\partial_{xx}^2 \ell(s,e^{\int_r^s\partial_Y f(u,Y_u,Z_u)du}D_rY_s)(e^{\int_r^s\partial_Y f(u,Y_u,Z_u)du}D_rZ_s)^2\big)ds\\
    &-\int_t^T \partial_x \ell(s,e^{\int_r^s\partial_Y f(u,Y_u,Z_u)du}D_rY_s) e^{\int_r^s\partial_Y f(u,Y_u,Z_u)du}D_rZ_s dB_s\\
    &+\int_r^T \partial_x \ell(s,e^{\int_r^s\partial_Y f(u,Y_u,Z_u)du}D_rY_s)e^{\int_r^s\partial_Y f(u,Y_u,Z_u)du} D_rk_sds.
\end{align*}
As $\partial_x\ell>0$ we let $\tilde Z^r_s=\partial_x \ell(s,e^{\int_r^s\partial_Y f(u,Y_u,Z_u)du}D_rY_s) e^{\int_r^s\partial_Y f(u,Y_u,Z_u)du}D_rZ_s$ and $\tilde Y^r_s = \ell(s,e^{\int_r^s\partial_Y f(u,Y_u,Z_u)du}D_rY_s)$ for $r\leq s\leq T$, so we can write
\begin{align*}
    &\tilde Y_t^r = \ell(T,e^{\int_r^T\partial_Y f(s,Y_s,Z_s)ds}D_r\xi)\\
    &+\int_t^T\big( \partial_x \ell(s,e^{\int_r^s\partial_Y f(u,Y_u,Z_u)du}D_rY_s)e^{\int_r^s\partial_Y f(u,Y_u,Z_u)du}D_r f(s,Y_s,Z_s)-\partial_t \ell(s,e^{\int_r^s\partial_Y f(u,Y_u,Z_u)du}D_rY_s)\\
    &+\partial_Z f(s,Y_s,Z_s)\tilde Z_s^r-\frac{1}{2}\frac{\partial_{xx}^2 \ell(s,e^{\int_r^s\partial_Y f(u,Y_u,Z_u)du}D_rY_s)}{\partial_x \ell(s,e^{\int_r^s\partial_Y f(u,Y_u,Z_u)du}D_rY_s)^2}(\tilde Z_s^r)^2\big)ds\\
    &-\int_t^T \tilde Z_s^r dB_s+\int_t^T \partial_x \ell(s,e^{\int_r^s\partial_Y f(u,Y_u,Z_u)du}D_rY_s)e^{\int_r^s\partial_Y f(u,Y_u,Z_u)du} D_rk_sds.
\end{align*}
Using Theorem 2.6 in \cite{kobyl} and verifying that $\tilde Y_t^r$ is bounded as $D_rY_t$ is bounded because $D_r\xi$ is bounded, we have $\tilde Y_t^r\geq \ubar Y_t^r$ where $(\ubar Y_t^r,\ubar Z_t^r)$ solves
\begin{align*}
    &\ubar Y_t^r = \ell(T,e^{\int_r^T\partial_Y f(s,Y_s,Z_s)ds}D_r\xi)\\
    &+\int_t^T\big( -\bar K de^{(T-r)\sup\partial_Y f}-M+\partial_Z f(s,Y_s,Z_s)\ubar Z_s^r-\frac{C}{2}(\ubar Z_s^r)^2\big)ds\\
    &-\int_t^T \ubar Z_s^r dB_s+\int_t^T Ke^{-(T-r)\sup|\partial_Y f|} D_rk_sds,
\end{align*}
assuming $\ubar Y_t^r$ is bounded.\\

If $C\leq 0$ we can take $C=0$ and we get directly
\begin{align*}
    &\E[\ell(t,Z_t)]\geq \E[\int_t^T \tilde\Gamma_t^s(-\bar K de^{(T-t)\sup\partial_Y f}-M+Ke^{-(T-t)\sup|\partial_Y f|} D_tk_s)ds +\tilde\Gamma_t^T\ell(T,e^{\int_t^T\partial_Y f(s,Y_s,Z_s)ds}D_t\xi)]
\end{align*}
and the same result as in the linear case applies. It is easy to verify that in this case $\ubar Y^t$ is bounded.
\end{proof}

\begin{remark}
Let us do an example. Let $\xi=-\int_0^T (T-t)dB_t$, $f=0$ and $l=Id$. Then $\E[D_t\xi]=t-T< 0$ if $t<T$ and $\E[D_T\xi]=0$. Take $k_t=e^{B_t-\frac{t}{2}}$. We get $\E[Z_t]=-(T-t)+\E[\int_t^Te^{B_s-\frac{s}{2}}ds]= 0$. To solve for $Y,Z$, we write $Y_t=-\int_0^T (T-s)dB_s-\int_t^T Z_s dB_s+\int_t^T e^{B_s-\frac{s}{2}}ds$ so by taking the expectation with respect to $\Fc_t$, we find $Y_t=e^{B_t-\frac{t}{2}}(T-t)-\int_0^t(T-s)dB_s$ so $Z_t=(e^{B_t-\frac{t}{2}}-1)(T-t)$, which can be negative. This implies that a comparison theorem cannot hold, as there exists a supersolution to the penalized BSDEs given in the preliminary example of Section \ref{section:impossibility} which is not bounded from below by the solutions of those BSDEs, which diverge. See \cite{buckd} and \cite{agram2019meanfield} for more discussions on this topic.
\end{remark}
\begin{remark}
Let $\xi=e^{B_T-\frac{T}{2}}$, $f(t,w,y,z)=e^{-B_t-\frac{t}{2}}$ and $\ell=Id-1$. According to the example of Section \ref{subsec::weakvsstrong} there is no supersolution with $Z_t\geq1$. However, note that $\E[D_t\xi]=1$ and by taking $k_t=e^{B_t-\frac{t}{2}}$, we get $\E[Z_t]=1+\E[\int_t^T(-e^{-B_s-\frac{s}{2}}+e^{B_s-\frac{s}{2}})ds]= 1$ there is a supersolution for the relaxed problem. We have in particular $Y_t=e^{B_T-\frac{T}{2}}-\int_t^T Z_s dB_s+\int_t^T (e^{-B_s-\frac{s}{2}}+e^{B_s-\frac{s}{2}})ds$ so by taking the expectation with respect to $\Fc_t$, we find $Y_t=e^{B_t-\frac{t}{2}}+(T-t)(e^{-B_t-\frac{t}{2}}+e^{B_t-\frac{t}{2}})$ and $Z_t=e^{B_t-\frac{t}{2}}+(T-t)(-e^{-B_t-\frac{t}{2}}+e^{B_t-\frac{t}{2}})$.
\end{remark}

\begin{remark}
Note that Proposition \ref{th::exi} enables us to construct supersolutions in the case where $\ell(t,.)$ is bounded from below by some increasing linear function which depends on $t$, \textit{i.e.} when $\ell(t,x)\geq a(t)x+b(t),$ for $a:[0,T]\longrightarrow \mathbb R^+$ and $b:[0,T]\longrightarrow  \mathbb R\setminus\{0\}$. Hence, $\mathbb E[\ell(t,Z_t)]\geq a(t)\mathbb E[Z_t]+b(t)$. Setting $\nu:=\frac{a}b,$ sufficient conditions ensuring the existence of supersolution of a $Z-$mean constrained BSDE can be deduced from Proposition \ref{th::exi}.\end{remark}

\subsection{Discussion on sufficient conditions.}

In this section, we give some simple sufficient conditions which imply those of Proposition \ref{th::exi} when $\ell$ is linear. Similar conditions can be obtained to relax the assumptions of Theorem \ref{th::thex} in a more general case.\\

First assume that $\partial_Z f=0$. Then, if the maps $\nu$, $t\mapsto\E[D_t\xi^+]$ and $t\mapsto\E[D_t\xi^-]$ are continuous at $t=T$, then the assumption $\nu(T)\leq \E[D_T\xi]$ is redundant with the boundedness assumption in Proposition \ref{th::exi}. If $\partial_Z f\neq 0$ the same result holds if $p_-$ and $p_+$ have limits at $t=T$ and $\nu$, $\E[|D_t\xi^+|^{\frac{1}{p_+(t)}}]^{p_+(t)}$ and $\E[|D_t\xi^-|^{p_-(t)}]^{\frac{1}{p_-(t)}}$ are continuous at $t=T$.\\

We now prove that the conditions are met under some continuity assumptions if the constraints are strictly satisfied at time $T$.
\begin{proposition}
Let $\nu$ be some real function on $[0,T]$ and $\ell(t,.)=Id(\cdot)-\nu(t)$. Assume also that $\nu$ is continuous and that assume that one of the following condition is satisfied,
\begin{enumerate}
    \item  $\partial_Z f=0$, the maps $t\mapsto\E[D_t\xi^+]$ and $t\mapsto\E[D_t\xi^-]$ are continuous, $\nu(T)< \E[D_T \xi]$;
    \item there exist two continuous map $p_-$ and $p_+$ on $[0,T)$ with left limits at $T-$ satisfying $p_-(T-)>1$ and $p_+(T-)>1$, $t\mapsto\E[|D_t\xi^+|^{\frac{1}{p_+(t)}}]^{p_+(t)}$ and $t\mapsto\E[|D_t\xi^-|^{p_-(t)}]^{\frac{1}{p_-(t)}}$ are continuous, and $\nu(T)<-\E[|D_T\xi^-|^{p_-(T-)}]^{\frac{1}{p_-(T-)}}+\E[|D_t\xi^+|^{\frac{1}{p_+(T-)}}]^{p_+(T-)}$, otherwise.
\end{enumerate}
Then the conditions of Proposition \ref{th::exi} are satisfied.
\end{proposition}
\begin{proof}
If $\partial_Z f=0$ it is enough to prove that 
\begin{align*}
        \limsup_{t\longrightarrow T}\frac{1}{T-t}\big[\nu(t)-(-e^{(T-t)\sup \partial_Y f }\E[D_t\xi^-]+e^{(T-t)\inf \partial_Y f }\E[D_t\xi^+])\big]<\infty.
    \end{align*}
This is satisfied as 
\begin{align*}
    \nu(t)-(-e^{(T-t)\sup \partial_Y f }\E[D_t\xi^-]+e^{(T-t)\inf \partial_Y f }\E[D_t\xi^+])\underset{t\longrightarrow T}{\longrightarrow}\nu(T)-c(T)<0.
\end{align*}
If $\partial_Z f\neq 0$ it is enough to prove that 
\begin{align*}
        &\limsup_{t\longrightarrow T}\frac{1}{T-t}\big[\nu(t)-(-e^{(T-t)\sup \partial_Y f }e^{\frac{1}{p_-(t)-1}\frac{\sup |\partial_Z f|^2}{2}(T-t)}\E[|D_t\xi^-|^{p_-(t)}]^{\frac{1}{p_-(t)}}\\
        &+e^{(T-t)\inf \partial_Y f }e^{\frac{p_+(t)}{1-p_+(t)}\frac{\sup |\partial_Z f|^2}{2}(T-t)}\E[|D_t\xi^+|^{\frac{1}{p_+(t)}}]^{p_+(t)})\big]<\infty.
    \end{align*}
This is satisfied as 
\begin{align*}
    &\nu(t)-(-e^{(T-t)\sup \partial_Y f }e^{\frac{1}{p_-(t)-1}\frac{\sup |\partial_Z f|^2}{2}(T-t)}\E[|D_t\xi^-|^{p_-(t)}]^{\frac{1}{p_-(t)}}\\
        &+e^{(T-t)\inf \partial_Y f }e^{\frac{p_+(t)}{1-p_+(t)}\frac{\sup |\partial_Z f|^2}{2}(T-t)}\E[|D_t\xi^+|^{\frac{1}{p_+(t)}}]^{p_+(t)})\\
    &\underset{t\longrightarrow T}{\longrightarrow}\nu(T)-(-\E[|D_t\xi^-|^{p_-(T-)}]^{\frac{1}{p_-(T-)}}+\E[|D_T\xi^+|^{\frac{1}{p_+(T-)}}]^{p_+(T-)})<0
\end{align*}
\end{proof}
Next we observe that if the constraint is satisfied at time $T$ but not strictly, a differentiability assumption is also sufficient.
\begin{proposition}\label{cor::phi}
Let $\nu$ be some real function on $[0,T]$ and $\ell(t,.)=Id(\cdot)-\nu(t)$. Assume also that one of the two following condition is satisfied
\begin{enumerate}
    \item $\partial_Z f=0$ and there exists some differentiable function $\phi:[0,T]\longrightarrow \mathbb R$ with $\phi(T)\leq 0$ such that $\nu(t)-(-e^{(T-t)\sup \partial_Y f }\E[D_t\xi^-]+e^{(T-t)\inf \partial_Y f }\E[D_t\xi^+])\leq \phi(t)$;
    \item there exists some differentiable function $\phi:[0,T]\longrightarrow \mathbb R$ with $\phi(T)\leq 0$ such that \begin{align*}
        &\nu(t)-(-e^{(T-t)\sup \partial_Y f }e^{\frac{1}{p_-(t)-1}\frac{\sup |\partial_Z f|^2}{2}(T-t)}\E[|D_t\xi^-|^{p_-(t)}]^{\frac{1}{p_-(t)}}\\
        &+e^{(T-t)\inf \partial_Y f }e^{\frac{p_+(t)}{1-p_+(t)}\frac{\sup |\partial_Z f|^2}{2}(T-t)}\E[|D_t\xi^+|^{\frac{1}{p_+(t)}}]^{p_+(t)})\leq \phi(t).
    \end{align*}
    .
\end{enumerate}
Then the conditions of Proposition \ref{th::exi} are satisfied.
\end{proposition}
\begin{proof}
In both cases an upper bound is given by
\begin{align*}
    \sup_{[0,T)} \frac{\phi(t)}{T-t}<+\infty
\end{align*}
as $t\mapsto\frac{\phi(t)}{T-t}$ is continuous on $[0,T)$ and $\frac{\phi(t)}{T-t}\sim_{t\rightarrow T} \frac{\phi(T)}{T-t}-\phi'(T)<+\infty$.
\end{proof}

\begin{remark}
Reciprocally under the conditions of Proposition \ref{th::exi} we can find a $\phi$ as in Proposition \ref{cor::phi}: take $\phi(t) = C(T-t)$ with $C$ the upper bound.
\end{remark}

\begin{remark}
The same sufficient conditions hold considering more general constraints function $\ell$.
\end{remark}

\subsection{Bilateral constraints}

In the case where $\partial_Z f=0$, the same argument as in Proposition \ref{th::exi} gives the following result if we have both upper and lower linear constraints.
\begin{proposition}
Let $\overline \nu$ and $\underline \nu$ be some real functions on $[0,T]$. Assume also that $\partial_Z f=0$, $\underline \nu(T)\leq \E[D_T \xi] \leq \overline \nu(T)$ and
there exists some c\`adl\`ag function $\nu:[0,T]$ such that,
    \begin{align*}
        &\frac{1 }{s(t)}[\overline \nu(t)-(-e^{(T-t)\inf \partial_Y f }\E[D_t\xi^-]+e^{(T-t)\sup \partial_Y f }\E[D_t\xi^+]+d e^{(T-t)\sup |\partial_Y f| }(T-t))]\\
        &\geq \nu(t)\\
        &\geq \frac{1 }{i(t)}[\underline \nu(t)-(-e^{(T-t)\sup \partial_Y f }\E[D_t\xi^-]+e^{(T-t)\inf \partial_Y f }\E[D_t\xi^+]-d e^{(T-t)\sup |\partial_Y f| }(T-t))]
    \end{align*}    
    if $\nu(t)\geq 0$ and
    \begin{align*}
        &\frac{1}{i(t)}[\overline \nu(t)-(-e^{(T-t)\inf \partial_Y f }\E[D_t\xi^-]+e^{(T-t)\sup \partial_Y f }\E[D_t\xi^+]+d e^{(T-t)\sup |\partial_Y f| }(T-t))]\\
        &\geq \nu(t)\\
        &\geq \frac{1}{s(t)}[\underline \nu(t)-(-e^{(T-t)\sup \partial_Y f }\E[D_t\xi^-]+e^{(T-t)\inf \partial_Y f }\E[D_t\xi^+]-d e^{(T-t)\sup |\partial_Y f| }(T-t))]
    \end{align*}
    if $\nu(t)\leq 0$.
Then there exists a solution to \eqref{eq::bsde} which satisfies $\underline \nu(t)\leq \E[Z_t]\leq \overline \nu(t)$ for all $0\leq t\leq T$.
\end{proposition}

\begin{remark}
Once again under continuity assumptions on $\underline \nu$, $\overline \nu$, $t\mapsto\E[D_t\xi^+]$ and $\E[D_t\xi^-]$ at $t=T$ the condition $\underline \nu(T)\leq \E[D_T \xi] \leq \overline \nu(T)$ is redundant.
\end{remark}

\subsection{On minimal super-solutions}\label{section:minimal}

\subsubsection{Characterization}
The following theorem is inspired by the discussion in Section 5 of \cite{briandElieHu} and formalizes the intuitive idea that a constraint in expectation makes the problem time inconsistent. We show in particular that a minimal supersolution is necessarily a solution. We prove the result in the case of linear constraints, but the theorem can be extended to concave constraints using similar arguments to those in Theorem \ref{th::thex}.

\begin{theorem}
Let the assumptions set out in Proposition \ref{th::exi} be satisfied and denote by $(Y,Z,K)$ a solution of \eqref{eq::bsde} satisfying the constraint \eqref{eq::constraint}. Then the following properties are equivalent
\begin{enumerate}[(i)]
    \item $(Y,Z,K)$ is a minimal supersolution;
    \item $K_T=0$ a.s..
\end{enumerate}
\end{theorem}
\begin{proof}
\textbf{(ii)$\implies$(i)}
This is a direct consequence of the comparison result for Lipschitz BSDEs, see \cite{elkpengquenez}.\\

\textbf{(i)$\implies$(ii)}
We proceed by contradiction by supposing that $\P(K_T>0)>0$. As $K_0=0$ and $\{K_T>0\}=\cup_{n=1}^{+\infty}\{K_T-K_{\frac{T}{n}}>0\}$, there is some $\varepsilon>0$ such that $\P(K_T-K_{\varepsilon}>0)>0$. Therefore $\E[\E[K_T-K_\varepsilon|\Fc_\varepsilon]]>0$ so that $\P(\E[K_T-K_\epsilon|\Fc_\epsilon]>0)>0$. We thus set $\Omega_+ = \{\E[K_T-K_\epsilon|\Fc_\epsilon]>0\}$.\\

Let $k_t\geq 0$ and $d\tilde K_t=k_t dt$. We aim at choosing $k_t$ so that $\tilde K\in\Ac^2$ and $k\in\Dc^{1,2}$. Fix $(\tilde Y,\tilde Z)$ so that $(\tilde Y,\tilde Z,\tilde K)$ is the unique solution to \eqref{eq::bsde} with $\tilde K$, using for example existence results in \cite{elkpengquenez}. Note that $f(t,\omega,0,0)+k_t\in\Hc^4$, enabling us to differentiate the BSDE \eqref{eq::bsde}. Following the same argument as in Proposition \ref{th::exi} we get
\begin{align*}
 \E[\tilde Z_t]\geq& -e^{(T-t)\sup \partial_Y f }\E[\tilde\Gamma_t^T D_t\xi^-]+e^{(T-t)\inf \partial_Y f }\E[\tilde\Gamma_t^T D_t\xi^+]-d e^{(T-t)\sup |\partial_Y f| }(T-t) \\
    &+\int_t^T\E[\Gamma_t^sD_tk_s]ds.
\end{align*}\\
Let $\kappa\geq 0$ and $\Omega^\kappa_\varepsilon$ be the event defined by
\[
\Omega^\kappa_\varepsilon:= \{ \lambda\in  C([0,T]),\; \sup_{0\leq t\leq \varepsilon} |\lambda(t)|\leq \kappa \}.
\]
Note that there is some $\kappa^*$ such that $0<\P(\Omega_+\cap \Omega^{\kappa^*}_\varepsilon)<1$. Let $p\geq 1$, $\tilde C>0$ and choose
\begin{align*}
    k_t = e^{\tilde C B_t-\frac{\tilde C^2 t^2}{2p}}\mathbf{1}_{t< \varepsilon} + e^{\tilde C B_t-\frac{C^2 t^2}{2p}}\mathbf{1}_{t\geq \varepsilon}X^*,
\end{align*}
where $X^*$ is a real random variable defined by 
\begin{align*}
    X^*(\omega)=(\kappa^* \vee \sup\{\kappa, \omega\in\Omega^\kappa_\varepsilon\})\wedge (\kappa^*+1)-\kappa^*=(\kappa^* \vee \sup\{| \omega_r|,0\leq r\leq \varepsilon\})\wedge (\kappa^*+1)-\kappa^*.
\end{align*}
Since $\{\omega\in\Omega^\kappa_\varepsilon\}$ is $\Fc_\varepsilon$-measurable for all $\kappa\in[0,+\infty)$, we obtain $\tilde K\in\Ac^2$. We begin to show that $Y_\varepsilon (\omega)>\tilde Y_\varepsilon(\omega)$ for $\omega\in \Omega_+\cap \Omega_\varepsilon^{\kappa^*}$. Following the argument of \cite[Theorem 2.5]{elkpengquenez} and since $Y_T-\tilde Y_T=0$, we have 
\begin{align*}
    Y_t-\tilde Y_t &= \int_t^T [f(s,Y_s,Z_s)-f(s,\tilde Y_s,\tilde Z_s)]ds-\int_t^T(Z_s-\tilde Z_s)dB_s+(K_T-K_t)-(\tilde K_T-\tilde K_t)\\
    &= \int_t^T[\Delta_y f(s)(Y_s-\tilde Y_s)+\Delta_z f(s)(Z_s-\tilde Z_s)]ds-\int_t^T(Z_s-\tilde Z_s)dB_s+(K_T-K_t)-(\tilde K_T-\tilde K_t)
\end{align*}
where $\Delta_y f(s)=(f(s,Y_s,Z_s)-f(s,\tilde Y_s,Z_s))/(Y_s-\tilde Y_s)$ if $Y_s-\tilde Y_s\neq 0$, and $\Delta_y f(s)=0$ otherwise, and $\Delta_z f(s)=(f(s,\tilde Y_s,Z_s)-f(s,\tilde Y_s,\tilde Z_s))/(Z_s-\tilde Z_s)$ if $Z_s-\tilde Z_s\neq 0$, and $\Delta_Z f(s)=0$ otherwise. By the assumptions on $f$, $\Delta_y f$ and $\Delta_z f$ are bounded so
\begin{align*}
    Y_\epsilon-\tilde Y_\epsilon = \E[\int_\epsilon^T e^{\int_t^s(\Delta_y f(u)-\frac{1}{2}|\Delta_z f(u)|^2)du+\int_t^s\Delta_z f (u)dB_u}(dK_s-d\tilde K_s)|\Fc_\epsilon].
\end{align*}
Let $\omega\in\Omega_+\cap\Omega^{\kappa^*}_\varepsilon$. Then $\tilde K_T(\omega)=\tilde K_\epsilon(\omega)$ by construction. Then
$$Y_\epsilon(\omega)-\tilde Y_\epsilon(\omega) =\E[\int_\epsilon^T e^{\int_t^s(\Delta f_y(u)-\frac{1}{2}\Delta f_z(u))du+\int_t^s\Delta f_z (u)dB_u}dK_s|\Fc_\epsilon](\omega).$$
As $Y_\epsilon(\omega)-\tilde Y_\epsilon(\omega)\leq 0$ implies $K_T(\tilde \omega)-K_\epsilon(\tilde \omega)=0$ for almost all $\tilde \omega\in \Omega$ such that $\tilde \omega(s)=\omega(s)$ for $0\leq s\leq \epsilon$, so $\E[K_T-K_\epsilon|\Fc_\epsilon](\omega)=0$ which is impossible given the definition of $\Omega_+$. Consequently,  $\tilde Y_\epsilon(\omega)<Y_\epsilon(\omega)$ for all $\omega\in\Omega_+\cap\Omega^{\kappa^*}_\varepsilon$ and any choice of $\tilde C$.\\

It remains to show that we can take some $\tilde C$ so that $\E[\tilde Z_t] \geq \nu(t)$, to get a supersolution $(\tilde Y, \tilde Z, \tilde K)$ such that $\P(\tilde Y_\epsilon<Y_\epsilon)\geq \P(\Omega_+\cap \Omega_{\kappa^*})>0$, to contradict that $Y$ is a minimal supersolution.\\

For $t\leq s<\varepsilon$, we have $\E[D_tk_s]=\tilde C$ with $p=1$ and 
\begin{align*}
    &\E[e^{\int_t^s \partial_Y f(r,\bar Y_r,\bar Z_r)dr +\partial_Z f(r,\bar Y_r,\bar Z_r)dB_r-\frac{\partial_Z f(r,\bar Y_r,\bar Z_r)^2}{2}dr}D_tk_s]\\
    &\geq e^{-T|\sup\partial_Y f|}e^{\frac{p}{1-p}\frac{\sup |\partial_Z f|^2}{2}T}\tilde C
\end{align*}
with $p>1$, using the reverse Hölder inequality.\\
For $t\vee \varepsilon\leq s$, using \cite{nualart2006malliavin} Section 2.1.7 and \cite{nakatsu} Lemma 2,\footnote{The Assumption $(iv)$ in this Lemma is not verified in our case where $\hat X=|B|$ but the proof of this lemma still works because, for almost all $r$ fixed, $s\mapsto D_r\hat X_s$ is constant on a small interval around the argsup of $X$.} we have $\mathbf{1}_{s\geq \epsilon}X^*(\omega)\in\Dc^{1,2}$ and
\begin{align*}
    D_t\mathbf{1}_{s\geq \epsilon}X^*(\omega)&=\mathbf{1}_{s\geq \varepsilon}\mathbf{1}_{\kappa^*<\sup\{|\omega_r|,0\leq r\leq \varepsilon\}<\kappa^*+1}\\
    &\hspace{1em}\times\mathbf{1}_{t\leq \argmax|\omega_r|,0\leq r\leq \varepsilon}(\mathbf{1}_{\argmax|\omega|=\argmax \omega}-\mathbf{1}_{\argmax|\omega|=\argmin \omega})\\
    &\geq -1.
\end{align*}
Hence, for any $0\leq t\vee \varepsilon\leq s\leq T$ we get for $p=1$
\begin{align*}
    \E[D_tk_s] &= \E[D_t \big(e^{\tilde CB_s-\frac{\tilde C^2 s^2}{2}}X^*\big)]\\
    &\geq \tilde C\E[e^{\tilde CB_s-\frac{\tilde C^2 s^2}{2}}\mathbf{1}_{\omega\not\in\Omega^{\kappa^*+1}_\varepsilon}]-\E[e^{\tilde CB_s-\frac{\tilde C^2 s^2}{2}}]\\
    &=\tilde C\E[e^{\tilde CB_s-\frac{\tilde C^2 s^2}{2}}\mathbf{1}_{\sup\{|B_r|, 0\leq r\leq \epsilon\}> \kappa^*+1}]-1\\
    &=\tilde C\P(\sup\{|B_r-\tilde C r|, 0\leq r\leq \varepsilon\}> \kappa^*+1)-1\\
    &\geq \tilde C\P(B_\varepsilon<\tilde C\varepsilon-\kappa^*-1)-1\underset{\tilde C\longrightarrow+\infty}{\longrightarrow} +\infty
\end{align*}
 and for $p>1$ we get
\begin{align*}
    &\E[e^{\int_t^s \partial_Y f(r,\bar Y_r,\bar Z_r)dr +\partial_Z f(r,\bar Y_r,\bar Z_r)dB_r-\frac{\partial_Z f(r,\bar Y_r,\bar Z_r)^2}{2p}dr}D_tk_s]\\
    &= \E[e^{\int_t^s \partial_Y f(r,\bar Y_r,\bar Z_r)dr +\partial_Z f(r,\bar Y_r,\bar Z_r)dB_r-\frac{\partial_Z f(r,\bar Y_r,\bar Z_r)^2}{2}dr}D_t e^{\tilde CB_s-\frac{\tilde C^2 s^2}{2p}}X^*(w)]\\
    &\geq e^{-T|\sup\partial_Y f|}\tilde C\E[e^{\int_t^s \partial_Z f(r,\bar Y_r,\bar Z_r)dB_r-\frac{\partial_Z f(r,\bar Y_r,\bar Z_r)^2}{2}dr}e^{\tilde CB_s-\frac{\tilde C^2 s^2}{2p}}\mathbf{1}_{w\not\in\Omega^{\kappa^*+1}_\varepsilon}]\\
    &-e^{T|\sup\partial_Y f|}\E[e^{\int_t^s \partial_Z f(r,\bar Y_r,\bar Z_r)dB_r-\frac{\partial_Z f(r,\bar Y_r,\bar Z_r)^2}{2}dr}e^{\tilde CB_s-\frac{\tilde C^2 s^2}{2p}}\mathbf{1}_{w\not\in\Omega^{\kappa^*}_\varepsilon}\mathbf{1}_{w\in\Omega^{\kappa^*+1}_\varepsilon}]\\
    &\geq e^{-T|\sup\partial_Y f|}\tilde C\E[e^{\int_t^s \partial_Z f(r,\bar Y_r,\bar Z_r)dB_r-\frac{\partial_Z f(r,\bar Y_r,\bar Z_r)^2}{2}dr}e^{\tilde CB_s-\frac{\tilde C^2 s^2}{2p}}\mathbf{1}_{\sup\{|B_r|, 0\leq r\leq \varepsilon\}> \kappa^*+1}]-e^{T|\sup\partial_Y f|}e^{\tilde C(\kappa^*+1)-\frac{\tilde C^2s^2}{2p}}\\
    &\geq e^{-T|\sup\partial_Y f|}e^{\frac{p}{1-p}\frac{\sup |\partial_Z f|^2}{2}(s-t)}\tilde C\P(\sup\{|B_r-\frac{\tilde C}{p}r|, 0\leq r\leq \varepsilon\}> \kappa^*+1)^p-e^{T|\sup\partial_Y f|}e^{\tilde C(\kappa^*+1)-\frac{\tilde C^2s^2}{2p}}\\
    &\geq e^{-T|\sup\partial_Y f|}e^{\frac{p}{1-p}\frac{\sup |\partial_Z f|^2}{2}T}\tilde C\P\big(B_\varepsilon<\frac{\tilde C}{p}\varepsilon-\kappa^*-1\big)^p-e^{T|\sup\partial_Y f|}e^{\tilde C(\kappa^*+1)-\frac{\tilde C^2s^2}{2p}}\\
    &\underset{\tilde C\longrightarrow+\infty}{\longrightarrow} +\infty
\end{align*}
 by using the reverse Hölder inequality.
Thus we can conclude, as we can write
\begin{align*}
 \E[\tilde Z_t]\geq& -e^{(T-t)\sup \partial_Y f }\E[\tilde\Gamma_t^T D_t\xi^-]+e^{(T-t)\inf \partial_Y f }\E[\tilde\Gamma_t^T D_t\xi^+]-d e^{(T-t)\sup |\partial_Y f| }(T-t) \\
    &+(T-t)c(\tilde C)
\end{align*}
for some function $c:\R\rightarrow\R$ with $c(\tilde C)\underset{\tilde C\rightarrow+\infty}{\rightarrow}+\infty$, and apply the same argument as in Proposition \ref{th::exi}.
\end{proof}

This result builds on some ideas of \cite{briandElieHu} on the time-inconsistency of problems with mean constraints. To obtain the existence of a minimal super-solution, stronger constraints are needed, for instance almost sure constraints, see \cite{bouchard2018bsde}, \cite{Peng1999MonotonicLT}. In the case with mean constraints on $Y$, another possibility is to restrain to a deterministic process $K$, but for constraints in $Z$ such super-solutions may not exist, and we were unable to prove the existence of a minimal one if at least one super-solution exists.

\subsubsection{Application: replication under $\beta$ law constraints}\label{sec::app}

Consider a stock market endowed with a bond with a deterministic interest rate $r$ and a stock with price dynamics given by
\begin{align*}
    dS_t=S_t(\mu_tdt+\sigma_tdB_t), 0\leq t\leq T,
\end{align*}
where the drift $\mu$ and the volatility $\sigma$ are two square integrable  predictable processes, such that $\sigma$ is lower bounded by some positive constant. Given an initial wealth $x$ and a square integrable claim $\xi$ at time $T$, a trader is looking for a hedging price. His wealth process is driven by a consumption-investment strategy $(\pi,K)$ with dynamics
\begin{align*}
    dX_t^{x,\pi,K} = r_tX_t^{x,\pi,K}dt+(\mu_t-r_t)\pi_tdt+\pi_t\sigma_tdB_t-dK_t.
\end{align*}
The admissible wealth processes satisfy the constraint
\begin{align*}
    \E[\pi_t]\geq \beta_t
\end{align*}
for some deterministic lower bound $\beta$ which penalizes portfolios which are too often contrarian but also allows for very contrarian positions in rare circumstances. The results of the paper show that there are admissible strategies, but there may not exist one with a deterministic investment process $K$. In addition the only strategy that might be optimal (in the sense that it is the preferred strategy at all time) if it is admissible is the one with no investment.

\section*{Acknowledgements}
The authors gratefully acknowledge the financial support of the ERC Grant 679836 Staqamof,
the Chaires Analytics and Models for Regulation, and Financial Risk. 
\bibliographystyle{apalike}
\bibliography{biblio.bib}

\appendix

\end{document}